\documentclass[12pt]{article}

\usepackage{amsmath}
\usepackage{latexsym}
\usepackage{amsfonts}
\usepackage{amssymb}
\usepackage{amscd}
\usepackage{theorem}

\newtheorem{theorem}{Theorem}[section]

\newtheorem{lemma}[theorem]{Lemma}
\newtheorem{proposition}[theorem]{Proposition}
\newtheorem{corollary}[theorem]{Corollary}

{\theorembodyfont{\rmfamily}
  \newtheorem{example}[theorem]{Example}

  \newtheorem{remark}[theorem]{Remark}
}

\newenvironment{proof}{	   
  \noindent
  \textbf{Proof.}}{
  \hfill $\Box$
  \vspace{3mm}
}

\numberwithin{equation}{section}

\newcommand{\N}{\mathbb{N}} 
\newcommand{\C}{\mathbb{C}} 
\newcommand{\D}{\mathbb{D}} 

\def\re{{\mathrm Re\,}}

\def\({\left(}       \def\){\right)}

\begin{document}

\title{A note on completeness of weighted normed spaces of analytic functions}

\author{Jos\'{e} Bonet and Dragan Vukoti\'c}

\date{{\small Dedicated to the memory of Pawe{\l} Doma\'nski (1959--2016)}}

\maketitle

\begin{abstract}
Given a non-negative weight $v$, not necessarily bounded or strictly positive, defined on a domain $G$ in the complex plane, we consider the weighted space $H_v^\infty(G)$ of all holomorphic functions on $G$ such that the product $v|f|$ is bounded in $G$ and study the question of when such a space is complete under the canonical sup-seminorm. We obtain both some necessary and some sufficient conditions in terms of the weight $v$, exhibit several relevant examples, and characterize completeness in the case of spaces with radial weights on balanced domains.
\end{abstract}

\vspace{.4cm}

\noindent \textbf{Authors' addresses:} \\

\noindent J.\ Bonet: Instituto Universitario de Matem\'{a}tica Pura y
Aplicada IUMPA,
\newline
Universitat Polit\`ecnica de Val\`encia,
\newline
E-46022 Valencia,
Spain
\newline
E-mail: jbonet@mat.upv.es \\

\noindent D.\ Vukoti\'c: Departamento de Matem\'aticas
\newline
Universidad Aut\'onoma de Madrid,
\newline
Facultad de Ciencias, M\'odulo 17
\newline
E-28049 Madrid, Spain
\newline
E-mail: dragan.vukotic@uam.es

\renewcommand{\thefootnote}{}

\newpage

\section{Introduction, Notation, and Motivation}

In this paper, as is usual, by a \textit{planar domain\/} we mean an open connected set in the complex plane $\C$. A \textit{weight} $v$ on a domain $G$ is a non-negative function $v:G \rightarrow [0,\infty[$. In general, it is not required that $v$ be bounded or strictly positive. Denote by $H(G)$ the algebra of all holomorphic (analytic) functions on $G$ and by $\tau_{co}$ the topology of uniform convergence on the compact subsets of $G$ (often also called the compact-open topology). The space $(H(G),\tau_{co})$ is a metrizable and complete locally convex space, \textit{i.e.\/}, a Fr\'echet space.
\par
The \textit{weighted space of holomorphic functions} $H^\infty_v(G)$ associated with $v$ is defined by
\begin{center}
 $H^\infty_v(G) := \{ f \in H(G) \ | \ \|f\|_v = \sup_{z \in G} v(z)
 |f(z)| < +\infty \}$
\end{center}
and is endowed with the natural \textit{seminorm} $\Vert f \Vert_v := \sup_{z \in G} v(z) |f(z)|.$ Spaces of this type, when $v$ is strictly positive and continuous, appear in the study of growth conditions of
analytic functions and have been investigated in various articles since the work of Shields and Williams; \textit{cf.\/} \cite{BBG}, \cite{BBT}, \cite{BDL}, \cite{L0}, \cite{L2}, \cite{SW} and the references therein.
\par
If $v$ is the constant function $1$, then $H^\infty_v(G)$ obviously coincides with the space $H^\infty(G)$ of all bounded holomorphic functions on $G$ endowed with the sup-norm $\Vert\cdot\Vert_\infty$.
In fact, in most cases considered in the literature, $v$ is continuous and strictly positive. In this case it is easy to check that the above weighted space is complete. It might be somewhat less obvious that if this is not required of $v$, then the space may fail to be complete (or even normed!), as will be seen in the examples given in this paper.
\par
The problem we consider in this note is the following: \textit{When
is the space $H^\infty_v(G)$ complete\/}? In other words, when is it a Banach space? We look for explicit conditions expressed in terms of the weight $v$. This is closely related to the question of boundedness of point evaluations. Proposition~\ref{prop01}, whose content should be intuitively clear to experts, gives a complete functional analytic characterization. Another general characterization, as one of the main results in the paper, is provided by Theorem~\ref{thm2}; it says that a bounded weight yields a complete space if and only if it can be replaced by a more regular weight that generates the same space. A natural Fr\'echet topology on the space $H^\infty_v(G)$, suggested to us by the referee,  is investigated in Proposition~\ref{proptoptau}. Several necessary as well as sufficient conditions, and also some relevant concrete examples, are given by Propositions~\ref{prop4improved}, \ref{prop3},  \ref{proplocalholes} and Corollaries \ref{cor2}, and \ref{cor4} and by Propositions~\ref{prop7} and \ref{prop6}. However, we are presently not able to give a complete intrinsic characterization of all weights $v$ for which $H^\infty_v(G)$ is complete.
\par
The situation is similar in other related function spaces. For example, the completeness of weighted Bergman spaces was studied by Arcozzi and Bj\"orn \cite{AZ}. They obtained complete characterizations when the weight $v(z) = \chi_E(z), z \in G,$ is the characteristic function $\chi_E$ of a subset $E$ of $G$ in \cite[Theorem 2.1]{AZ}. Partial results concerning weighted Bergman spaces $A^p_{\mu}(G), 1 \leq p < \infty,$ for a positive Borel measure $\mu$ on $G$ are given in \cite[Section~5]{AZ}. This research was taken up by Bj\"orn in a different direction \cite{Bj}. The closely related question of completeness of weighted Bloch spaces was investigated by Nakazi \cite{N}.
\par

\section{Functional analytic approach}\label{sect1}
We begin this section with some basic results. Our approach is based on functional analysis. Given a weight $v$ on $G$, throughout the paper we will use the following notation:
$$
 E_v := \{ z \in G \ | \ v(z) > 0 \}.
$$
For most of the ``reasonable'' weights our weighted space is complete and one certainly expects it to be at least normed. However, even this is not always the case.
\begin{proposition}\label{prop1} Let $v: G \rightarrow [0,\infty[$ be a weight on a planar domain $G$. Then the space $H^\infty_v(G)$ is normed if and only if $E_v$ is not a discrete set (that is, it has a limit point in $G$).
\end{proposition}
\begin{proof}
If $E_v$ has a limit point in $G$, then the seminorm $\Vert . \Vert_v$ is a norm by the uniqueness principle for holomorphic functions. Conversely, if $E_v$ does not have a limit point in $G$, we can apply the Weierstrass interpolation theorem (see, \textit{e.g.\/} \cite[Theorem~3.3.1]{BG}) to produce a non-zero holomorphic function $f \in H(G)$ such that $f(z)=0$ for each $z \in E_v$. Then $\Vert f \Vert_v = 0$ and $f \neq 0$, hence $\Vert \cdot \Vert_v$ is not a norm.
\end{proof}
\par
Given a seminormed space $(X,p)$, the associated normed space is defined by $(\tilde{X},\tilde{p}):=(X/ker(p), \tilde{p})$, with $\tilde{p}(x+ ker(p)):=p(x)$, which is easily seen to be a well-defined norm on $X/ker(p)$.
\begin{proposition}\label{prop1bis} Let $v: G \rightarrow [0,\infty[$ be a weight on a planar domain $G$. If the set $E_v$ defined above does not have a limit point in $G$, then the normed space associated with $H^\infty_v(G)$ is isomorphically isometric to a weighted Banach $\ell_{\infty}$ space.
\end{proposition}
\begin{proof}
If $E_v$ does not have  a limit point in $G$, then it is a discrete sequence in $G$. Let us write $E_v:=\{z_n \}_n$ and define $w(n):=v(z_n), n \in \N,$ $w:=(w(n))_n$ and $\ell_{\infty}(w):=\{x=(x_n)_n \in \C^{\N} \ | \ ||x||_w:=\sup_{n \in \N} w(n)|x_n| < \infty \}$. The linear map $\Phi: H^\infty_v(G) \rightarrow \ell_{\infty}(w)$ given by
$\Phi(f):=(f(z_n))_n$ is well defined, satisfies $||\Phi(f)||_w = ||f||_v$ for each $f \in H^\infty_v(G)$, is surjective by the Weierstrass interpolation theorem and its kernel
coincides with the kernel of $\|\cdot\|_v$. This completes the proof.
\end{proof}
\par
Now that this elementary issue has been settled, we turn to the completeness question. We first require a lemma.
\begin{lemma}\label{lem01}
Let $v: G \rightarrow [0,\infty[$ be a weight on a planar domain $G$. If the space $H^\infty_v(G)$ is normed, then the inclusion map $J: H^\infty_v(G) \rightarrow (H(G),\tau_{co})$ has closed graph.
\end{lemma}
\begin{proof}
Let $(f_j)_j$ be a sequence in $H^\infty_v(G)$ such that $f_j \rightarrow f$ in $H^\infty_v(G)$ and $f_j \rightarrow g$ in $(H(G),\tau_{co})$ as $j \rightarrow \infty$. In particular, $f_j(z) \rightarrow f(z)$ as $j \rightarrow \infty$ for each $z \in E_v$ and $f_j(z) \rightarrow g(z)$ as $j \rightarrow \infty$ for each $z \in G$. Then $f$ and $g$ are two holomorphic functions on $G$ which coincide on the set $E_v$, that has a limit point in $G$ by Proposition~\ref{prop1}. By the uniqueness principle for holomorphic functions,  $f=g$ on $G$ and we are done.
\end{proof}
\par
For a point $z \in G$, we denote by $\delta_z: H(G) \rightarrow \C$ the (linear) point evaluation functional $\delta_z(f):=f(z), \ f \in H(G)$, as well as its restriction to $H^\infty_v(G)$. When $H^\infty_v(G)$ is normed, the norm of its dual space $H^\infty_v(G)'$ will be denoted by $\Vert. \Vert'_v$. The following result summarizes a result an expert would expect: the completeness of our space is essentially equivalent to the boundedness of the point evaluation functionals. Regarding condition (iv) below (uniform boundedness of point evaluations on compact sets), it should be pointed out that this property is in turn equivalent to their boundedness at each point when the space is Banach, in view of the uniform boundedness principle.
\par
\begin{proposition}\label{prop01}
Assume that the space $H^\infty_v(G)$ is normed. The following conditions are equivalent:
\begin{itemize}
\item[(i)] The space $H^\infty_v(G)$ is a Banach space.

\item[(ii)] The inclusion map $J: H^\infty_v(G) \rightarrow (H(G),\tau_{co})$ is continuous. (Equivalently, every sequence in $H^\infty_v(G)$ bounded in the norm is a normal family.)

\item[(iii)] The closed unit ball $B^\infty_v$ of $H^\infty_v(G)$ is bounded in $(H(G),\tau_{co})$.

\item[(iv)] For each $z \in G$, the point evaluation functional  $\delta_z \in H^\infty_v(G)'$ and, moreover, $\sup_{z \in K} \Vert \delta_z \Vert'_v < \infty$ for each compact subset $K$ of $G$.

\end{itemize}
\end{proposition}
\begin{proof}
Condition (i) implies condition (ii) as a consequence of Lemma~\ref{lem01} and the closed graph theorem for Fr\'echet spaces.
\par
To prove that condition (ii) implies condition (i), fix a Cauchy sequence $(f_j)_j$ in $H^\infty_v(G)$. By assumption (ii), there exists $f \in H(G)$ such that $(f_j)_j$ converges to $f$ uniformly on the compact subsets of $G$. On the other hand,
$$
 \forall \varepsilon >0 \ \exists J \ \forall j,k \geq J \ \forall z \in G \ : \ \ v(z)|f_j(z) -f_{k}(z)| < \varepsilon.
$$
If $v(z)=0$ then $v(z)|f_j(z) -f(z)|=0, j \geq J$, and if $v(z)>0$, letting $k \rightarrow \infty$, $v(z)|f_j(z) -f(z)| \leq \varepsilon$ for all $ j \geq J$. This implies, for $\varepsilon =1$, $v(z)|f(z)| \leq 1 + \Vert f_{J} \Vert_v$ for each $z \in G$ and $f \in H^\infty_v(G)$. Moreover, for arbitrary $\varepsilon$, we have that $f_j \rightarrow f$ in $H^\infty_v(G)$ as $j \rightarrow \infty$.
\par
Thus, conditions (i) and (ii) are equivalent. Clearly, conditions (ii) and (iii) are also equivalent.
\par
We will now show the equivalence of (ii) and (iv). Suppose first that condition (ii) holds. Since $\delta_z \in (H(G),\tau_{co})^\prime$ for each $z \in G$, it follows that $\delta_z \in H^\infty_v(G)^\prime$ for each $z \in G$. Moreover, given a compact subset $K$ of $G$ there exists $C>0$ such that
$$
 \sup_{z \in K} |f(z)| \leq C \sup_{z \in G} v(z)|f(z)|
$$
for each $f \in H^\infty_v(G)$. This implies $\Vert \delta_z \Vert'_v \leq C$ for each $z \in K$. Hence condition (iv) follows.
\par
Suppose now that condition (iv) holds. Fix a compact set $K$ in $G$ and set $M:=\sup_{z \in K} \Vert \delta_z \Vert'_v$. If $f \in H^\infty_v(G)$ satisfies $\Vert f \Vert_v \leq 1$, then $|f(z)| \leq \Vert \delta_z \Vert'_v \leq M$ for each $z \in K$. This implies $\sup_{z \in K} |f(z)| \leq M \Vert f \Vert_v$ for each $f \in H^\infty_v(G)$, and the inclusion map $J: H^\infty_v(G) \rightarrow (H(G),\tau_{co})$ is continuous. This proves (ii).
\end{proof}
\begin{remark}
As a consequence of Ptak's version of the closed graph theorem \cite[Theorem 4, page 301]{H}, if $H^\infty_v(G)$ is a non-complete normed space, then it is not barrelled, \textit{i.e.\/}, there are weak-$^\ast$ bounded sets in the topological dual which are not norm bounded.
\end{remark}

\begin{theorem}\label{thm2}
Let $v: G \rightarrow [0,\infty[$ be a bounded weight on a planar domain $G$ such that the space $H^\infty_v(G)$ is normed. The space $H^\infty_v(G)$ is complete if and only if there is a bounded, continuous, strictly positive weight $\tilde{v}$ on $G$ such that $H^\infty_v(G)=H^\infty_{\tilde{v}}(G)$.
\end{theorem}
\begin{proof}
It is well-known that if $\tilde{v}$ is a bounded, continuous, strictly positive weight $\tilde{v}$ on $G$, then the space $H^\infty_{\tilde{v}}(G)$ is a Banach space. We prove the converse. To do this we follow ideas of \cite{BBT}. By assumption there is $M>0$ such that $0 \leq v(z) \leq M$ for each $z \in G$. Hence, the constant function $f_0(z):=1/M, z \in G,$ belongs to $H^\infty_v(G)$ and $\Vert f_0 \Vert_v \leq 1$.
\par
For each $z \in G$ we have $\delta_z \in H^\infty_v(G)'$ by Proposition~\ref{prop01}, and $\Vert \delta_z \Vert'_v \geq |f_0(z)| =1/M >0$. We set
$$
\tilde{v}(z):=1/\Vert \delta_z \Vert'_v, \ \ z \in G.
$$
By the previous estimate $0<\tilde{v}(z) \leq M$ for each $z \in G$. Moreover, $v(z) \leq \tilde{v}(z)$ for each $z \in G$. In fact, the inequality is obvious if $v(z)=0$. If $v(z)>0$ and
$g \in H^\infty_v(G)$ satisfies $\Vert g \Vert_v \leq 1$, then $|g(z)| \leq 1/v(z)$. This implies $1/\tilde{v}(z) \leq 1/v(z)$. Thus $v(z) \leq \tilde{v}(z)$. This implies, in particular, that $H^\infty_{\tilde{v}}(G) \subset H^\infty_v(G)$ with a continuous, norm decreasing inclusion.
\par
Now $H^\infty_{\tilde{v}}(G) = H^\infty_v(G)$ holds isometrically. Indeed, if $f \in H^\infty_v(G)$ with $\Vert f \Vert_v \leq 1$, then $\tilde{v}(z)|f(z)| \leq 1$ for each $z \in G$. Therefore $f \in H^\infty_{\tilde{v}}(G)$ and $\Vert f \Vert_{\tilde{v}} \leq 1$. This implies $H^\infty_v(G) \subset H^\infty_{\tilde{v}}(G)$ with a norm decreasing inclusion.
\par
It remains to prove that the weight $\tilde{v}$ is continuous. Indeed, the map $\Delta: G \rightarrow H^\infty_v(G)', \ \Delta(z):= \delta_z$ is well defined and locally bounded since every $z \in G$ has a compact neighborhood and the conclusion follows from condition (iv) in Proposition~\ref{prop01}. Now, for each $f \in H^\infty_v(G) \subset H^\infty_v(G)''$, the map $T_f \circ \Delta: G \rightarrow \C, z \rightarrow f(z),$ is holomorphic on $G$. By \cite[Theorem 1]{G} the vector valued mapping $\Delta: G \rightarrow H^\infty_v(G)'$ is holomorphic, hence continuous for the dual norm $\Vert . \Vert'_v$  on $H^\infty_v(G)'$. Since the norm is continuous, it follows that the function given by $\tilde{v}(z) = 1/\Vert \Delta(z) \Vert'_v$ is continuous. This completes the proof.
\end{proof}
\par
\begin{remark}
(1) A weight $v$ on $G$ is bounded if and only if the constant function $1$ belongs to $H^\infty_v(G)$ if and only if every bounded analytic function on $G$ belongs to $H^\infty_v(G)$. In this case $H^\infty_v(G)$ is non-trivial. (As is customary, we will say that a vector space is \textit{non-trivial\/}  if it contains a non-zero vector.)
\par
(2) As is usual, from now on we write
$$
  B(z_0,r):=\{ z \in \C \ | \ |z-z_0| < r \}
$$
to denote the open disk of radius $r$ centered at $z_0$. Assume that there exists a point $z_0 \in E_v$ with $B(z_0,r_0)\subset E_v$ for some $r_0 >0$ and such that the function $w(r):= \sup \{ 1/v(z) \ | \ |z-z_0| < r \}, 0 < r < r_0$, satisfies $\lim_{r \rightarrow 0} w(r)/r^n = 0$ for each $n \in \N$. Then $H^\infty_v(G) =  \{ 0 \}$. Indeed, if $f \in H^\infty_v(G)$, and $f(z)= \sum_{n=0}^{\infty} a_n (z-z_0)^n$ is the Taylor series expansion of $f(z)$ in $B(z_0,r_0)$, then the Cauchy estimates imply that $|a_n| \leq ||f||_v w(r)/r^n$ for each $0 < r < r_0$, which yields $a_n = 0$ for each $n \in \N$.
\par
(3) Given any positive integer $n$, the space $H^\infty_v(G)$ can be $n$-dimensional, at least in the case when $G=\C$. Take, for example, $v(z):= \min(1, |z|^{1/2-n})$, $z \in \C$. In this case it follows again from the Cauchy estimates that  $H^\infty_v(\C)$ consists only of the polynomials of degree at most $n-1$ since $|f(z)|\le C |z|^{n-1/2}$ for $|z|>1$.
\par
(4) Assume that $v$ is locally bounded, \textit{i.e.\/}, for each $z \in G$ there is $r(z)>0$ such that $B(z,r(z)) \subset G$ and $\sup \{v(\zeta) \ | \ \zeta \in B(z,r(z)) \} < \infty$.
If $H^\infty_v(G)\neq \{ 0 \}$, then $||\delta_z||'_v >0$ for each $z \in G$ with $\delta_z \in H^\infty_v(G)'$. Indeed, if there is a non-zero function $f_0 \in H^\infty_v(G)$ such that $f_0(z_0)=0$ and $k$ is the order of the zero $z_0$ of $f_0$, then it is easy to see that the function $g_0:=f_0(z)/(z-z_0)^k, z \in G,$ belongs to $H^\infty_v(G)$ and $g_0(z_0) \neq 0$.
\par
(5) If $v$ is locally bounded and  the dimension of $H^\infty_v(G)$ is at least $2$, then $H^\infty_v(G)$ separates points of $G$. Indeed, let $z_1$ and $z_2$ be two different points in $G$. If $ker(\delta_{z_1})$ in $H^\infty_v(G)$ is $\{ 0 \}$, then the dimension of $H^\infty_v(G)$ is $0$ or $1$, because $ker(\delta_{z_1})$ is a hyperplane of  $H^\infty_v(G)$. Otherwise, there is a non-zero function $f \in H^\infty_v(G)$  such that $f(z_1)=0$. If $f(z_2) \neq 0$, we are done. In case $f(z_2) = 0$ and $k$ is the order of the zero $z_2$ of $f(z)$, then it is enough to take $g(z):=f(z)/(z-z_2)^k \in H^\infty_v(G)$ to separate $z_1$ and $z_2$. This argument is adapted from the proof of \cite[Lemma 4]{BV}.
\end{remark}
\par\smallskip
It appears natural to consider the following topology $\tau$ on $H^\infty_v(G)$ that combines the convergence on the compact subsets of $G$ with the uniform convergence on $E_v$ induced by $1/v$. Select a fundamental sequence $(K_n)_n$ of compact subsets of $G$ and define the sequence of norms
$$
||f||_n := \sup_{z \in K_n} |f(z)| + ||f||_v, \ \  f \in H^\infty_v(G), \ \ n \in \N.
$$
It is easy to see that $(H^\infty_v(G), \tau)$ is a Fr\'echet space, that the topology $\tau$ is finer than the compact open topology $\tau_{co}$ and also finer than the topology $\tau_v$ induced by the seminorm $\|\cdot\|_v$. In fact, it is the coarsest topology that is finer than these two topologies. Our next result collects some elementary facts about this topology. Recall that a locally convex topology $\sigma$ on a space $X$ is \textit{normable\/} if there is a norm $|\cdot|$ in the space such that the topology $\sigma$ coincides with the topology induced by the norm $|\cdot|$. The space $(X,\sigma)$ is normable if and only if there is a $\sigma$-continuous (semi)norm $p$ on $X$ such that for every continuous seminorm $q$ on $(X,\sigma)$ there is a constant $C>0$ such that $q(x) \leq C p(x)$ for each $x \in X$.
\begin{proposition}\label{proptoptau}
Let $v: G \rightarrow [0,\infty[$ be a weight on a planar domain $G$.
\begin{itemize}
\item[(i)] The topology $\tau$ is coarser than $\tau_v$ if and only if the two topologies coincide (equivalently, if $(H^\infty_v(G),\|\cdot\|_v)$ is a Banach space).

\item[(ii)] The topology $\tau$ is normable if and only if there is a compact set $K$ in $G$ such that the weight $w(z):=v(z)+\chi_K(z), z \in G,$ where $\chi_K$ is the characteristic function of $K$, makes $H^\infty_w(G)$ into a Banach space.

\item[(iii)] If $E_v$ is contained in a compact subset of $G$ and $v$ is bounded, then $\tau$ coincides with the compact open topology $\tau_{co}$ on $H^\infty_v(G)$. Under the assumptions that $H^\infty_v(G)$ contains the polynomials and $G$ is simply connected, the converse also holds: if $\tau=\tau_{co}$ on $H^\infty_v(G)$, then there is a compact set $K$ in $G$ such that $E_v \subset K$.

\item[(iv)] If the topology $\tau$ is normable, then $E_v$ is not discrete.
\end{itemize}
\end{proposition}
\begin{proof}
(i) If $\tau$ is coarser than $\tau_v$, then the two topologies coincide and  $(H^\infty_v(G),\|\cdot\|_v)$ is a Banach space, because $\tau_v$  is  Hausdorff and complete since it coincides with $\tau$. Conversely, if $(H^\infty_v(G),\|\cdot\|_v)$ is a Banach space, then $\tau_v$ is finer than $\tau_{co}$ by Proposition~\ref{prop01}. This implies that $\tau_v$ is finer than $\tau$.
\par
(ii) The topology $\tau$ is normable if and only if there is $m$ such that for all $n$ there is $C_n>0$ with $||f||_n \leq C_m ||f||_m$ for each $f \in H^\infty_v(G)$. We select $K=K_m$ and define $w$ as in the statement. Clearly $H^\infty_v(G)=H^\infty_w(G)$ and $||f||_w = \sup_{z \in G} w(z)|f(z)| \leq ||f||_m \leq 2 ||f||_w$ for each $f \in H^\infty_v(G)$. Therefore, $\tau$ is normable if and only if $\tau=\tau_w$. The conclusion now follows from part (i).
\par
(iii)  The topology $\tau$ is always finer than $\tau_{co}$ on $H^\infty_v(G)$. Assume that there is $m \in \N$ such that $E_v \subset K_m$ and that there is $M>0$ such that $v(z) \leq M$ for all $z \in G$. Then $||f||_v  \leq M \sup_{z \in K_m} |f(z)|$ for each $f \in H^\infty_v(G)$. This implies $||f||_n \leq (M+1) \sup_{z \in K_n} |f(z)|$ for each $f \in H^\infty_v(G)$ and each $n \geq m$. Therefore $\tau$ is coarser than $\tau_{co}$ on $H^\infty_v(G)$.

Now assume that $H^\infty_v(G)$ contains the polynomials, $G$ is simply connected and $\tau=\tau_{co}$ on $H^\infty_v(G)$. Then there is a compact set $K$ in $G$ and there is $C>0$ such that $\sup_{z \in G} v(z)|f(z)| \leq C \sup_{z \in K} |f(z)|$ for each $f \in H^\infty_v(G)$. We may assume that $\C \setminus K$ is connected. Proceeding by contradiction, suppose that there is $z_0 \in E_v \setminus K$. We have $|f(z_0)| \leq (C/v(z_0)) \sup_{z \in K} |f(z)|$ for each $f \in H^\infty_v(G)$. By Runge's theorem $K$ is polynomially convex and there is a polynomial $g$ such that $|g(z_0)| > \alpha > \sup_{z \in K} |g(z)|$. Define $h_n:=(g/\alpha)^n, n \in \N$. Each $h_n$ is a polynomial, hence it belongs to $H^\infty_v(G)$ by assumption. We have $|h_n(z_0)| \leq (C/v(z_0)) \sup_{z \in K} |h_n(z)| \leq C/v(z_0)$ for each $n \in \N$. But
$\lim_{n \rightarrow \infty} |h_n(z_0)| = \infty$. This is a contradiction.
\par
(iv) Assume that $E_v$ is discrete in $G$ and define $X:=\{f \in H^\infty_v(G) \ | \ f(z)=0 \ {\rm for \ all } \ z \in E_v \}$. The space $X$ (and accordingly $H^\infty_v(G)$) is infinite dimensional. Indeed, take another discrete set $F$ in $G$ disjoint with $E_v$. We can apply the Weierstrass interpolation theorem to find a linearly independent sequence $(f_n)_n$ of analytic functions on $G$ which vanish on $E_v$. Clearly this sequence is contained in $X$. On the other hand, the topology $\tau$ restricted to $X$ coincides with the
restriction to $X$ of the topology $\tau_{co}$ of uniform convergence on the compact subsets of $G$. To see this, just compare the norms on elements of $X$. If $(H^\infty_v(G),\tau)$ is normable, so is the space $(X,\tau)$. But $(X,\tau)$ is a closed subspace of the  space $(H(G), \tau_{co})$. By Montel's theorem the bounded subsets of $(X,\tau_{co})=(X,\tau)$ are relatively compact. By a theorem of Riesz the normed space $(X,\tau_{co})=(X,\tau)$ must be finite dimensional. This is a contradiction.
\end{proof}
\par
As a consequence of Proposition~\ref{proptoptau}, if $H^\infty_v(G)$ is a Banach space (in particular in the classical case when $v$ is continuous and strictly positive), then the topology $\tau$ coincides with $\tau_v$. On the other hand, if $v$ is a weight on a simply connected domain $G$ in $\C$ such that $E_v$ is an infinite discrete subset of $G$ and $H^\infty_v(G)$ contains the polynomials, then $\tau$ is not normable by Proposition \ref{proptoptau} (iv) and it is strictly finer than the topology $\tau_{co}$ of uniform convergence on the compact subsets of $G$ by Proposition \ref{proptoptau} (iii). See also Corollary~\ref{cor0} below.

\section{Function theoretic approach}\label{sect2}
Our approach in this section (kindly suggested by the referee) is based on function theory. We begin with the following lemma.
\begin{lemma}\label{holext1}
Let $v$ be a weight on the planar domain $G$ such that $E_v$ is not discrete. Then the normed space $H^\infty_v(G)$ is complete if and only if every function $f:E_v \rightarrow \C$ for which there is a sequence $(f_n)_n$ in $H^\infty_v(G)$ such that $v(z)|f_n(z)-f(z)| \rightarrow 0$ uniformly on $E_v$ as $n \rightarrow \infty$ has a (necessarily unique) holomorphic extension to $G$.
\end{lemma}
\begin{proof}
If $H^\infty_v(G)$ is complete and $(f_n)_n$ is a sequence in $H^\infty_v(G)$ such that $v(z)|f_n(z)-f(z)| \rightarrow 0$ uniformly on $E_v$ as $n \rightarrow \infty$ for some
function $f:E_v \rightarrow \C$, then $(f_n)_n$ is a Cauchy sequence in $H^\infty_v(G)$ and there is $g \in H^\infty_v(G)$ such that $||f_n - g||_v \rightarrow 0$ as $n \rightarrow \infty$. The function $g$ is a holomorphic extension of $f$ to $g$. Conversely, let $(f_n)_n$ is a Cauchy sequence in $H^\infty_v(G)$. There is $f:E_v \rightarrow \C$ such that $v(z)|f_n(z)-f(z)| \rightarrow 0$ uniformly on $E_v$ as $n \rightarrow \infty$. Our assumption now implies that there is $g \in H(G)$ such that $g(z)=f(z)$ for all $z \in E_v$. Clearly $g \in H^\infty_v(G)$ and the sequence $(f_n)_n$ converges to $g$ in $H^\infty_v(G)$.
\end{proof}
\par\smallskip
In the sequel, as is usual, we will denote the boundary of a set $A$ by $\partial A$.
\begin{proposition} \label{prop4improved}
Let $v: G \rightarrow [0,\infty[$ be a  weight on a planar domain $G$. If $H^\infty_v(G)$ is a non-trivial Banach space, then the boundary $\partial G$ is contained in the closure $\overline{E}_v$ of $E_v$ in $\C$.
\end{proposition}
\begin{proof}
Assume, on the contrary, that there is $z_0 \in \partial G \setminus \overline{E}_v$. Let $r>0$ be such that $|z-z_0| \geq r$ for all $z \in E_v$. Select a non-zero function
$h \in H^\infty_v(G)$. There is $z_1 \in G$ with $|z_1-z_0|< r$ such that $h(z_1) \neq 0$. We have that
$$
f(z):= \frac{z_1-z_0}{z-z_1}= \sum_{n=1}^{\infty} \(\frac{z_1-z_0}{z-z_0} \)^n
$$
uniformly on $|z-z_0|\geq |z_1-z_0| + \varepsilon$ for all $\varepsilon >0,$ and therefore uniformly on $E_v$. Then the functions
$$
g_n(z):= \sum_{k=1}^{n} \(\frac{z_1-z_0}{z-z_0} \)^k \ , \ \ \ n \in \N,
$$
are holomorphic on $G$ and converge to $f$ on $E_v$. Note that $h g_n \in H^\infty_v(G)$ since $g_n$ is bounded on $E_v$. We have
$$
\sup_{z \in E_v} v(z)|(hg_n)(z)-(hf)(z) \leq ||h||_v \sup_{z \in E_v} |g_n(z) - f(z)| \rightarrow 0,
$$
as $n \rightarrow \infty$. By Lemma~\ref{holext1} the function $hf:E_v \rightarrow \C$ has a unique holomorphic extension to $G$. However, $hf$ has a pole at $z_1$.
This is a contradiction.
\end{proof}

The idea of working with a non-zero element $h$ of $H^\infty_v(G)$ in the above proof can be found in Gaier \cite[p.~151]{Ga}. It avoids the assumption that $v$ is bounded which was needed in our original proof.
\begin{corollary} \label{cor1}
Let $v: G \rightarrow [0,\infty[$ be a weight on a planar domain $G$ (other than the plane itself) such that $H^\infty_v(G)$ is normed and non-trivial. \begin{itemize}
\item[(1)] If $E_v$ is contained in a convex closed proper subset $A$ of $G$, then $H^\infty_v(G)$ is not a Banach space.
\item[(2)] If the closure of $E_v$ in $\C$ is a compact subset of  $G$, then $H^\infty_v(G)$ is not a Banach space.
\end{itemize}
\end{corollary}
\begin{proof}
This is a direct consequence of Proposition~\ref{prop4improved}.
\end{proof}
\par
As a consequence of Corollary~\ref{cor1} one easily deduces the following example that one expects intuitively: the weight $v(z)=\max  \{0,\re z\}$ is continuous on the unit disk $\D$, vanishes in the left-hand half of the disk, it is strictly positive in the remaining open right semi-disk, and $H^\infty_v(\mathbb{D})$ is not a Banach space.
\par
\begin{corollary} \label{cor0}
Let $v: G \rightarrow [0,\infty[$ be a weight on a planar domain $G$ (other than the plane itself) such that $H^\infty_v(G)$ is normed and non-trivial. If the boundary $\partial G$ is not contained in the closure $\overline{E}_v$ of $E_v$ in $\C$, then the topology $\tau$ of Proposition~\ref{proptoptau} on the space $H^\infty_v(G)$ is not normable.
\end{corollary}
\begin{proof}
Assume, on the contrary,  that  $(H^\infty_v(G),\tau)$ is normable. By Proposition~\ref{proptoptau} there is a compact set $K$ in $G$ such that the weight $w(z):=v(z)+\chi_K(z), z \in G,$ where $\chi_K$ is the characteristic function of $K$, satisfies that $H^\infty_w(G)$ is a Banach space. The boundary $\partial G$ is not contained in the closure $\overline{E_w}=\overline{E_v \cup K}$ of $E_w$ in $\C$, since is not contained in the closure $\overline{E}_v$ of $E_v$ in $\C$ and $K$ is a compact subset of $G$. Proposition~\ref{prop4improved} implies that $H^\infty_w(G)$ is not a Banach space; a contradiction.
\end{proof}
\par\smallskip
Let $A$ be subset of a domain $G$ in $\C$. We recall that the \textit{holomorphically convex hull of $A$ in $G$} is the set
$$
Hco(A):= \{z \in G \ | \ |f(z)| \leq \sup_{\zeta \in A} |f(\zeta)|\,, \ \forall f \in H(G) \}.
$$
Every domain $G$ in $\C$ is holomorphically convex in the sense that for each compact set $K$ in $G$ the holomorphic convex hull $Hco(K)$ is compact and contained in $G$; \textit{cf.\/} \cite{Hor}. With this concept at hand, we can obtain the following complement of Proposition~\ref{prop4improved} for bounded weights which includes the case $G=\C$. It implies, for example, that $H^\infty_v(\C)$ is not a Banach space if $v$ is a bounded weight on $\C$ such that $E_v$ is relatively compact.
\par
\begin{proposition} \label{prop3}
Let $v: G \rightarrow [0,\infty[$ be a bounded weight on a planar domain $G$. If $H^\infty_v(G)$ is a Banach space, then $G$ coincides with the holomorphic convex hull $Hco(E_v)$ of $E_v:=\{z \in G \ | \ v(z)>0 \}$.
\end{proposition}
\begin{proof}
We give a proof by contradiction. Assume there exist a point $z_0 \in G$ and a function $g_0 \in H(G)$ such that
$$
|g_0(z_0)| > \alpha > \sup_{\zeta \in E_v} |g_0(\zeta)|
$$
for some $\alpha>0$. Set $g:=g_0/\alpha \in H(G)$. Then $|g(z_0)|>1$ and $|g(\zeta)| \leq 1$ for each $\zeta \in E_v$.
\par
We show that the sequence $(g^k)_k$ is bounded in $H^\infty_v(G)$. Since $v$ is bounded, there is $M>0$ with $v(z) \leq M$ for each $z \in G$. If $z \notin E_v$, then $v(z)|g^k(z)| =0$ for each $k \in \N$. On the other hand, if $z \in E_v$, then $v(z)|g(z)^k| \leq v(z) \leq M$. Hence $\sup_{k \in \N} \sup_{z \in G} v(z) |g(z)^k| \leq M$.
\par
By assumption $H^\infty_v(G)$ is a Banach space, hence $(g^k)_k$ is a bounded sequence in $(H(G),\tau_{co})$ by Proposition~\ref{prop01} (ii). However, $|g(z_0)^k| \rightarrow \infty$ as $k \rightarrow \infty$, which is absurd.
\end{proof}
\par
Lemma~\ref{holext1} can also be used to get positive results.
\par
\begin{proposition} \label{proplocalholes}
Let $v: G \rightarrow [0,\infty[$ be a  weight on a planar domain $G$. Suppose that for each $z \in G$ there is a bounded open set $U \subset G$ such that $z \in U$, $\partial U \subset E_v$, and $v$ is bounded away from $0$ on $\partial U$. Then $H^\infty_v(G)$ is a Banach space
\end{proposition}
\begin{proof}
Let $f:E_v \rightarrow \C$ be a function and let $(f_n)_n$ be a sequence in $H^\infty_v(G)$ such that $v(z)|f_n(z) - f(z)| \rightarrow 0$ uniformly on $E_v$. By Lemma~\ref{holext1} it suffices to prove that $f$ has a holomorphic extension to $G$. Let $U \subset G$ be a non-empty bounded open subset such that $\partial U \subset E_v$ and $v$ is bounded away from $0$ on $\partial U$. Then $v(z)|f_n(z) - f(z)| \rightarrow 0$ uniformly on $\partial U$ and hence $|f_n(z) - f_m(z)| \rightarrow 0$ as $n,m \rightarrow \infty$ uniformly on $\partial U$. The maximum modulus principle implies that $(f_n)_n$ converges uniformly on $\overline{U}$ to a function $f_U$ that is holomorphic on $U$, continuous on $\overline{U}$ and that coincides with $f$ on $\overline{U} \cap E_v$.
\par
Now, let $U,V$ be two bounded open sets with $U \cap V \neq \emptyset$ such that $\partial U \cup \partial V \subset E_v$ and $v$ is bounded away from $0$ on $\partial U$ and $\partial V$. Then $\partial(U \cap V) \subset \partial U \cup \partial V \subset E_v$ and $v$ is bounded away from $0$ on $\partial(U \cap V)$. Thus the three holomorphic functions $f_U$, $f_V$ and $f_{U \cap V}$ are defined. Since $f_U$ and $f_V$ agree on $\partial(U \cap V)$, they agree on $U \cap V$. This shows that if $U \subset G$ is a bounded set containing $z$ such that $\partial U \subset E_v$ and $v$ is bounded away from $0$ on $\partial U$, by setting $g(z):= f_U(z), z \in G,$ we define in a unique way a holomorphic function $g$ on $G$ that coincides with $f$ on $E_v$.
\end{proof}
\begin{corollary}\label{cor2}
Let $v: G \rightarrow [0,\infty[$ be a continuous weight on a planar domain $G$ such that $G\setminus E_v$ is discrete, \textit{i.e.\/}, the zeros of $v$ are isolated. Then $H^\infty_v(G)$ is a Banach space.
\end{corollary}
\begin{proof}
This is a direct consequence of Proposition~\ref{proplocalholes}.
\end{proof}
\par
As a consequence of Corollary~\ref{cor2}, if $F \in H(G)$ is a non-zero holomorphic function on a planar domain $G$ and $v(z):=|F(z)|, z \in G,$ then $H^\infty_v(G)$ is a Banach space. This example is mentioned in \cite[Example 3]{N}.
\begin{corollary}\label{cor3}
Let $v: G \rightarrow [0,\infty[$ be a weight on a planar domain $G$ such that $\overline{G\setminus E_v}$ is a compact  subset of $G$.
If $\inf_{z \in K} v(z) >0$ for each compact subset $K \subset G \setminus (\overline{G\setminus E_v})$, then $H^\infty_v(G)$ is a Banach space.
\end{corollary}
\begin{proof}
We select a bounded open set $V$ whose closure $\overline{V}$ is a compact subset of $G$ and such that the compact subset  $\overline{G\setminus E_v}$ of $G$ is contained in $V$. The boundary $\partial V$ of $V$ is also a compact subset of $G$ and $\inf_{z \in \partial V} v(z) >0$ by assumption. If $z \in \overline{G\setminus E_v}$, we take the open set $V$ to get $z \in V$ and $v$ is bounded away from $0$ on $\partial V$. If $z \notin \overline{G\setminus E_v}$, it is enough to take an open disk $U$ centered at $z$ whose closure does not meet $\overline{G\setminus E_v}$. By assumption $\inf_{z \in \partial U} v(z) >0$. The conclusion follows from Proposition~\ref{proplocalholes}.
\end{proof}
\par
The assumptions of Corollary \ref{cor3} are satisfied if $v$ is the characteristic function of a subset $A$ of $G$ such that $G \setminus A$ is a compact subset of $G$.
\begin{corollary}\label{cor4}
Let $G=\D$ (resp.\  $G=\C$). Let $v$ be a bounded radial weight on $G$. The space $H^\infty_v(G)$ is a Banach space if and only if $E_v$ is not compact in $G$ or, equivalently, if and only if there is an increasing sequence $(r_k)_k$ in $]0,1[$ tending to $1$ (resp.
$(r_k)_k$ in $]0,\infty[$ tending to $\infty$) such that $v(r_k)>0$ for each $k \in \N$.
\par
In particular, if $v(z):=|F(|z|)|, z \in G,$ for a non-zero function $F \in H(G)$, then $H^\infty_v(G)$ is a Banach space.
\end{corollary}
\begin{proof}
This is a direct consequence of Proposition~\ref{proplocalholes} and Corollary~\ref{cor1} (2) (for $G \neq \C$) and Proposition~\ref{prop3} (for $G=\C$).
\end{proof}
\par
Let $v$ be the weight on $\D$ defined by $v(z):= a_n >0$ if $|z|=1-(1/n)$, and $v(z)=0$ otherwise. Then $H^\infty_v(G)$ is a Banach space by Corollary~\ref{cor4}. Observe that the sequence $(a_n)_n \subset ]0,\infty[$ need not be bounded. Similar examples can be obtained by replacing $\D$ by $\C$ and $1-(1/n)$ by $n$, $n \in \N$.
\begin{proposition} \label{prop7}
Let $F \in H(G)$ be a non-zero function on a planar domain $G$.
Define $v(z):=0$ if $F(z)=0$ and $v(z):=1/|F(z)|$ if $F(z) \neq 0$. Then $H^\infty_v(G)$ is a Banach space that coincides with the set of all $f \in H(G)$ such that there is $C=C(f)>0$ with $|f(z)| \leq C |F(z)|$ for each $z \in G$.
\end{proposition}
\begin{proof}
The weight $v$ is in general unbounded and not continuous, but it has isolated zeros. Then $H^\infty_v(G)$ is a normed space by Proposition~\ref{prop1} and a Banach space by Proposition~\ref{proplocalholes}.
\par
Now we prove the other part of the statement. If $f \in H(G)$ satisfies $|f(z)| \leq C |F(z)|$ for each $z \in G$, then $f(z)=0$ whenever $F(z)=0$. Moreover, $v(z)|f(z)| \leq C v(z)|F(z)|$ for each $z \in G$, hence $f \in H^\infty_v(G)$ and $\Vert f \Vert_v \leq C$. On the other hand, if $f \in H^\infty_v(G)$ satisfies $\Vert f \Vert_v = D>0$, then $|f(z)| \leq D|F(z)|$ if $F(z) \neq 0$. Since both $f$ and $F$ are continuous and the zeros of $F$ are isolated, this inequality implies that $f(z)=0$ whenever $F(z)=0$. Therefore $f \in H^\infty_v(G)$ if and only if there is $C>0$ with $|f(z)| \leq C |F(z)|$ for each $z \in G$.
\end{proof}
\begin{remark}
Observe, for the weight $v$ on $G$ considered in Proposition~\ref{prop7}, that we have $\Vert \delta_z \Vert'_v = 0$ if $F(z)=0$ (since each $f \in H^\infty_v(G)$ vanishes on such $z \in G$), and $\Vert \delta_z \Vert'_v = |F(z)|$ if $F(z) \neq 0$.
In fact, $F \in H^\infty_v(G)$ and $\Vert F \Vert_v =1$, thus $\Vert \delta_z \Vert'_v \geq |F(z)|$. Moreover, if $f \in B^\infty_v$, then $|f(z)| \leq |F(z)|$, which yields
$\Vert \delta_z \Vert'_v \leq |F(z)|$. Therefore $\Vert \delta_z \Vert'_v = |F(z)|$ for each $z \in G$.
It follows that the weight $\tilde{v}(z):=1/\Vert \delta_z \Vert'_v$ associated with this particular weight $v$ and constructed in the proof of Theorem~\ref{thm2}, is in general not defined on the whole set $G$ and is unbounded on the set on which it is defined. This shows that the assumption that the weight $v$ is bounded cannot be omitted  in Theorem~\ref{thm2}.
\end{remark}
\par\smallskip
The following result complements Proposition~\ref{proplocalholes}. It allows us to construct examples of Banach spaces $H^\infty_v(\D)$ for a weight $v$ which, in each circle of radius $1-(1/n)$, $n \in \N$, takes a strictly positive value $\alpha_n$ on a dense subset $D_n$ of the circle and is $0$ outside the union of the sets $D_n, n \in \N$.
\begin{proposition}\label{prop6}
Let $v: G \rightarrow [0,\infty[$ be a weight on a planar domain $G$.
Assume that for each compact set $K \subset G$ there is a compact set $L$ such that $K \subset L \subset G$ and a number $\alpha >0$ such that
$$
\partial L \subset \overline{ \{z \in G \ | \ v(z) \geq \alpha \} }.
$$
Then $H^\infty_v(G)$ is a Banach space.
\end{proposition}
\begin{proof}
We first prove that the assumption implies that for each compact set $K \subset G$ there is a compact set $M$ such that $K \subset M \subset G,$ and there is a positive constant $\alpha$ such that $ K$ is contained in the holomorphic convex hull $Hco \overline{ \{z \in M \ | \ v(z) \geq \alpha \} }$ of $\overline{ \{z \in M \ | \ v(z) \geq \alpha \}}$. To see this, fix a compact set $K \subset G$. We apply the assumption to find a compact set $L$ containing $K$ and $\alpha$ such that $\partial L \subset \overline{ \{z \in G \ | \ v(z) \geq \alpha \} }$. If $G=\C$, take $d=1$, and if $G \neq \C$, set $d:={\rm dist}(L, \C\setminus G)$. The set $M:=\{x \in \C \ | \ {\rm dist}(x,L) \leq d/2 \}$ is compact, contained in $G$ and $L \subset M$. Define $S:= \overline{ \{z \in M \ | \ v(z) \geq \alpha \}} \subset M$. We show that $\partial L \subset S$. Indeed, if $z \in \partial L$, there is a sequence $(x_j)_j \subset G$ such that $v(x_j) \geq \alpha$ for each $j \in \N$ and $x_j \rightarrow z$ as $j \rightarrow \infty$. There is $J \in \N$ such that, for $j \geq J$, ${\rm dist}(x_j,L) \leq |x_j - z| < d/2$. If $j\geq J$, then $x_j \in M$ and $v(x_j) \geq \alpha$, which means $x_j \in S$. This implies $z \in S$.
\par
Now, if $z \in K \subset L$ and $f \in H(G)$, we can apply the maximum principle to get
$$
|f(z)| \leq \sup_{\zeta \in L} |f(\zeta)| = \sup_{\zeta \in \partial L} |f(\zeta)| \leq \sup_{\zeta \in S} |f(\zeta)|.
$$
This implies that $K \subset Hco(S)$.
\par
Now we proceed to prove that the closed unit ball $B_v^\infty$ of $H^\infty_v(G)$ is bounded in $(H(G),\tau_{co})$; the conclusion will follow from Proposition~\ref{prop01}.
\par
Given a compact set $K \subset G$ we apply the first part of the proof to find the compact set $M \subset G$ and $\alpha>0$. Set $R:= \{z \in M \ | \ v(z) \geq \alpha \}$. If $f \in B_v^\infty$ and $z \in K$, then $z \in Hco(\overline{R})$. Thus, since $f$ is continuous and $\overline{R} \subset M \subset G$, we get
$$
|f(z)| \leq \sup_{\zeta \in \overline{R}}|f(\zeta)| = \sup_{\zeta \in R} |f(\zeta)|\,, \quad \forall z\in K\,.
$$
If $\zeta \in R$, then $v(\zeta) \geq \alpha$. This implies $\alpha |f(\zeta)| \leq v(\zeta)|f(\zeta)| \leq 1$ for each $\zeta \in R$. Hence $\sup_{\zeta \in R} |f(\zeta)| \leq 1/\alpha$. Therefore $|f(z)| \leq 1/\alpha$ for each $z \in K$ and each $f \in B_v^\infty$,
which implies that $B_v^\infty$ is bounded in $(H(G),\tau_{co})$.
\end{proof}
\par\smallskip
The reader should notice that there are various situations which are not covered by the above results. We include some of them in this final part of the paper.
\begin{proposition}\label{prop-int}
Let $v: G \rightarrow [0,\infty[$ be a weight on a planar domain $G$ with the property that, for every point $z$ in $G$, there exist positive numbers $p_z$ and $R_z$ satisfying $R_z<\mathrm{dist\,}(z,\partial G)$ and
$$
 \int_0^{2\pi} \frac{1}{(v(z + R_z e^{i\theta}))^{p_z}} \frac{d \theta}{2 \pi} < +\infty.
$$
Then $H^\infty_v(G)$ is a Banach space.
\end{proposition}
\begin{proof}
It suffices to show that the closed unit ball $B_v^\infty$ of $H^\infty_v(G)$ is bounded in $(H(G),\tau_{co})$. The conclusion will then follow from Proposition \ref{prop01}. To this end, fix a compact subset $K$ of $G$. Obviously, $K\subset \cup_{z\in K} B(z,\frac{1}{2}R_z)$ and by compactness we can select finitely many points $z_1,...,z_J \in K$ so that
\begin{equation}
 K\subset \bigcup_{j=1}^J B(z_j,\frac{1}{2}R_{z_j})\,.
 \label{subcover}
\end{equation}
\par
Set $R_j:=R_{z_j}$ and $r_j:=\frac{1}{2}R_{z_j}, j=1,...,J$. Then  $0<r_j<R_j$ for all $j$ and  $K \subset \bigcup_{j=1}^{J} B(z_j,r_{j})$.
Since $dm(\theta)=d\theta/(2\pi)$ is a probability measure on $[0,2\pi]$, an elementary application of H\"older's inequality shows  that for a fixed $v$, $z$, and $R$ the integral means
$$
 \(\int_0^{2\pi} \frac{1}{(v(z + R e^{i\theta}))^{p}} \frac{d \theta}{2 \pi}\)^{1/p}
$$
increase as $p\in (0,\infty)$ increases. Thus, for a given compact subset $K$ and the corresponding points $z_j$ and values $R_j$, $r_j$, and $p_j$ from the assumptions of the statement, by choosing $p=\min\{p_1,\ldots,p_J\}$ it follows that
\begin{equation}
 \int_0^{2\pi} \frac{1}{(v(z_j + R_j e^{i\theta}))^{p}} \frac{d \theta}{2 \pi} =: I_j(K) <+\infty, \qquad j=1,2,\ldots,J.
 \label{int-p}
\end{equation}
If $z \in \overline{B(z_j,r_{j})}$ for some $j=1,...,J,$ then, for each $f$ in the unit ball $B_v^\infty$ of $H_v^\infty (G)$, we get
$$
 |f(z)|^p \leq \frac{R_{j}+r_{j}}{R_{j}-r_{j}} \int_0^{2\pi} |f(z_j + R_{j} e^{i\theta})|^p \frac{d \theta}{2 \pi} \leq \frac{R_{j}+r_{j}}{R_{j}-r_{j}} I_j(K).
$$
The second inequality is clear, since $v(\zeta)|f(\zeta)| \leq 1$ for each $\zeta \in G$. To justify the first one, first observe that $u(\zeta):=|f(\zeta)|^p, \zeta \in G,$ is a non-negative, continuous and subharmonic function and then apply the Poisson integral inequality to $u$: for $0<r\leq r_{j}$ and all $t \in [0,2\pi]$, we have
$$
 u(z_j+re^{it}) \leq \frac{R_{j}+r}{R_{j}-r} \int_0^{2\pi} u(z_j + R_{j} e^{i\theta}) \frac{d \theta}{2 \pi}.
$$
Accordingly,
$$
 \sup_{z \in \overline{B(z_j,r_{j})}} |f(z)| \leq \( \frac{R_{j}+r_{j}}{R_{j}-r_{j}} I_j(K)\)^{1/p}.
$$
Therefore
$$
 \sup_{z \in K} |f(z)| \leq \max \left\{ \( \frac{R_{j}+r_{j}}{R_{j}-r_{j}} I_j(K)\)^{1/p}, \ j=1,...,J \right\}
$$
for each  $f \in B_v^\infty$, and $B_v^\infty$ is bounded in $(H(G),\tau_{co})$.
\end{proof}
\begin{example}
Let $q>0$ and $v(z)=|{\rm Re}\,z|^q, z \in \D$. Then the normed space $H^\infty_v(\D)$ is complete. (This solves an interesting question posed to us by the referee.)
\par
To see this, we check that the condition of Proposition~\ref{prop-int} is satisfied. Indeed, select $p>0$ such that $0 < pq < 1$; this $p$ will serve for all points $z\in\D$. We will  choose the value of $R_z$ depending on the location of the point $z$ in $\D$ with respect to the zero set of $v$:
\par
(1) If $z$ is not on the imaginary axis, we can pick $R$ sufficiently small so that $\overline{B(z;R)}$ does not intersect this axis and hence the function $v^{-p}$ is bounded both from above and from below on this disk.
\par
(2) If $z=yi$ is any point on the imaginary axis ($0\le |y|<1$), we select $R\in [0,1[$ with $|yi+R e^{i \theta}| < 1$ for each $\theta \in [0, 2 \pi]$. Note that
$$
 v(z+R e^{i \theta})= |\textrm{Re\,}(yi+R e^{i \theta})|^q= R^q |\cos \theta|^q.
$$
Hence
$$
 \int_0^{2\pi} \frac{1}{(v(yi+R e^{i\theta}))^p} d \theta = \frac{1}{R^{pq}} \int_0^{2\pi} \frac{d\theta}{|\cos \theta |^{pq}},
$$
which is finite since the function $\frac{1}{|\cos \theta|^{\beta}}$ is integrable in $[0,2 \pi]$ if $0 < \beta < 1$.
\end{example}
\begin{example}
Let $v$ be a  weight on $\D$ such that there is a strictly increasing sequence $(r_n)_n$ of positive numbers tending to $1$ such that for each $n$ there is $a_n>0$ such that
$v(r_n e^{i\theta}) \geq a_n$ almost everywhere in $[0,2 \pi]$. Then the normed space $H^\infty_v(\D)$ is complete. Indeed, define a radial weight $w$ on $\D$ by setting $w(r_n):=\min(a_n,1), n \in \N,$ and $0$
otherwise. By Corollary \ref{cor4}, $H^\infty_w(\D)$ is complete, so that the closed unit ball $B^{\infty}_w$ of $H^\infty_w(\D)$ is bounded in $(H(\D), \tau_{co})$. Now let $f$ belong to the closed unit ball $B^{\infty}_v$ of $H^\infty_v(\D)$. Then, for each $n \in \N$,
$$
a_n |f(r_n e^{i \theta})| \leq v(r_n e^{i \theta}) |f(r_n e^{i \theta})| \leq 1
$$
almost everywhere in $[0, 2 \pi]$. Since $f$ is continuous, we have $w(z)|f(z)| \leq 1$ for each $z \in \D$, so that $f \in B^{\infty}_w$. Thus $B^{\infty}_v \subset B^{\infty}_w$, and $B^{\infty}_v$ is also bounded in $(H(\D), \tau_{co})$. The conclusion follows from Proposition \ref{prop01}.
\end{example}
\begin{remark}\label{rem_final}
Let $w$ be a continuous weight on $G$ such that $H^\infty_w(\D)$ is a Banach space. Let $v$ be a weight on $G$ such that $\{z \in G | v(z)=w(z) \}$ is dense in $G$
(note that this implies that $E_v$ is not discrete, so that $H^\infty_v(\D)$ is normed). Then $H^\infty_v(\D)$ is actually a Banach space. Indeed, if $f \in B^{\infty}_v$, then $v(z)|f(z)| \leq 1$
for all $z \in G$. The assumption and the continuity of both $f$ and $w$ imply that $w(z)|f(z)| \leq 1$ for all $z \in G$. Thus $f \in B^{\infty}_w$. The conclusion follows again from Proposition \ref{prop01}. In particular, if $w$ is a continuous strictly positive weight on $\D$ and $v$ is defined as follows:
$$
 v(z)=
\begin{cases}
 0, \qquad \ \mathrm{if\ } \mathrm{Re\,} z = 0
\\[\jot]
 w(z), \quad \mathrm{if\ } \mathrm{Re\,} z \neq 0,
\end{cases}
$$
then $H^\infty_v(\D)$ is a Banach space.
\end{remark}
\par\smallskip

\vspace{1cm}
\par
\noindent \textbf{Acknowledgements.} The authors are indebted to the anonymous referee for a patient and careful reading of the manuscript and for making a number of constructive and useful suggestions, as well as for raising some interesting questions. This made it possible to substantially improve the paper. Amongst the suggestions made by the referee, we mention the idea of considering the topology $\tau$ described in Section \ref{sect1}, Proposition~\ref{prop4improved}, Proposition~\ref{proplocalholes} and Remark \ref{rem_final}, as well as their proofs.
\par
The first author was partially supported by MTM2013-43540-P and MTM\-2016-76647-P by MINECO/FEDER-EU and GVA Prometeo II/2013/013. The second author was partially supported by the MINECO/FEDER-EU grant MTM2015-65792-P. Both authors were partially supported by Thematic Research Network MTM2015-69323-REDT, MINECO, Spain.


\bibliographystyle{amsplain}

\end{document}